\documentclass[12pt]{amsart}
\usepackage{hyperref}
\usepackage{amsmath}
\usepackage{amsthm}
\usepackage{mathrsfs}
\usepackage{amssymb}
\usepackage{tikz-cd}
\usepackage{comment}
\usepackage{mathtools}
\usepackage{manfnt}
\calclayout
\newcommand{\PP}{\mathbb{P}}

\newcommand{\CC}{\mathbb C}

\newcommand{\OO}{\mathcal O}
\newcommand{\OC}{\mathcal C}
\newcommand{\ra}{\rightarrow}

\newcommand{\cvg}{\mathrm{cvg}}
\newcommand{\scvg}{\mathrm{scvg}}
\newcommand{\Sym}{\mathrm{Sym}}

\newtheorem{theorem}{Theorem}[section]

\newtheorem{lemma}[theorem]{Lemma}
\newtheorem{proposition}[theorem]{Proposition}

\theoremstyle{definition}

\newtheorem*{remark}{Remark}
\newtheorem{example}[theorem]{Example}

\newtheorem{definition}[theorem]{Definition}

\newtheorem{manualtheoreminner}{Proposition}
\newenvironment{manualProposition}[1]{%
  \renewcommand\themanualtheoreminner{#1}%
  \manualtheoreminner
}{\endmanualtheoreminner}

\title{Covering gonalities of complete intersections in positive characteristic}
\author{Geoffrey Smith}
\date{}
\begin{document}
\maketitle
\begin{abstract}
We define the covering gonality and separable covering gonality of varieties over arbitrary fields,  generalizing the definition given by Bastianelli-de Poi-Ein-Lazarsfeld-Ullery for complex varieties. We show that over an arbitrary field a smooth multidegree $(d_1,\ldots,d_k)$ complete intersection in $\mathbb{P}^N$ has separable covering gonality at least $d-N+1$, where $d=d_1+\cdots+d_k$. We also show that the very general such hypersurface has covering gonality at least $\frac{d-N+2}{2}$.
\end{abstract}
\section{Introduction}
Let $C$ be a smooth projective curve over a field $k$. A classical measure of the failure of $C$ to be rational is the \emph{gonality} of $C$, $\text{gon}_k(C)$, defined as the minimum degree of a finite map $C\ra \PP^1_k$. Equivalently, $\text{gon}_k(C)$ is the minimal degree of an extension $K(C)/k(t)$, where $t$ is some transcendental element of $K(C)$. The significance of the gonality has inspired multiple generalizations to higher dimensional varieties. The first introduced is the \emph{degree of irrationality} of an irreducible variety $X$ of dimension $n$ over a field $k$, defined in \cite{MH82} as the quantity
\[
\text{irr}_k(X)=\min \{[K(X):k(x_1,\ldots,x_n)]\vert \{x_1,\ldots,x_n\} \text{ a transcendence basis for } K(X)\}.
\]

More recently, Bastinallei, de Poi, Ein, Lazarsfeld and Ullery \cite{BDELU17} introduced several other measures of irrationality for varieties over the complex numbers, including the \emph{covering gonality}, which is defined for a complex variety $X$ as the minimum gonality of a curve contained in $X$ passing through a general point. Several articles have computed or bounded the covering gonality for specific classes of complex varieties \cite{Bas12, BCFS19, Voi18, Mar19}. In this article we study the covering gonality of varieties over arbitrary fields.  
\begin{definition}\label{cvgdef}
Let $X$ be an irreducible proper variety of dimension $n$ over a field $k$. The \emph{covering gonality} of $X$ over $k$, denoted $\cvg(X)$, is the minimal $e$ such that there exists a diagram of irreducible varieties over $k$,
\begin{equation}\label{cvg}
\begin{tikzcd}
\OC \arrow[r, "f"] \arrow[d, "\pi"]
& X \\
\PP^1\times B 
\end{tikzcd},
\end{equation}
where $f$ is a dominant generically finite map and $\pi:\OC\ra \PP^1 \times B$ is a finite surjective morphism of degree $e$. The \emph{separable covering gonality} of $X$, denoted $\scvg_k(X)$, is defined as the minimal $e$ such that such a diagram exists with $f$ separable.
\end{definition}
\begin{remark}
It is essential for this definition that there be a single algebraic family of curves dominating $X$. If we merely require that for any geometric point $p$ of $X$ there exists a gonality $e$ curve passing through $p$, \cite{BT05} shows there would exist some non-uniruled varieties defined over $\mathbb F_p$ having covering gonality 1 over $\overline{\mathbb F}_p$, namely the non-supersingular Kummer surfaces. If $k$ is uncountable and $X$ is projective, this issue does not come up, by Proposition \ref{properties} (5).
\end{remark}

Of these two measures, the separable covering gonality is by far the better behaved invariant, and many results proven for the covering gonality of complex varieties hold for the separable covering gonality as is. For instance, we observe that the proof of Theorem A of \cite{BDELU17}  is sufficient to prove the following slight generalization.
\begin{proposition}[\cite{BDELU17}, Theorem A]\label{BDELU}
Let $X\subset \PP_k^{N}$ be a smooth geometrically irreducible complete intersection of $k$ hypersurfaces of degree $d_1,\ldots,d_k$ with $\sum_i d_i\geq N+1$. Then we have
\[
scvg_k(X)\geq \sum_{i} d_i-N+1
\]
\end{proposition}
The covering gonality is more poorly behaved, as the following example of Shioda and Katsura illustrates.
\begin{example}[\cite{SK79}, Theorem III]\label{shioda}
Let $k$ be a field of positive characteristic $p$. Let $r,d$ be positive integers with $r$ even and $d\geq 4$. Then, if there exists an integer $v$ such that 
\[
p^v\equiv -1 \pmod{d},
\]
then the Fermat hypersurface cut out by the homogeneous polynomial $x_0^d+\ldots+x_{r+1}^d=0$ in $\PP^{r+1}_k$ is unirational.
\end{example}
So there are smooth hypersurfaces $X$ with $d-n$ arbitrarily large such that $\cvg(X)=1$ if $k$ has positive characteristic, so Proposition \ref{BDELU} has no chance of being true without modification for the covering gonality. Instead, we have the following bound.
\begin{theorem}\label{mainCompleteIntersections}
Let $X\subset \PP^N$ be a very general complete intersection variety of  multidegree $(d_1,\ldots,d_k)$ over an uncountable field $k$. Then 
\[
cvg_k(X)\geq \frac{\sum_id_i-N+2}{2}.
\]
\end{theorem}
This will be a consequence of a somewhat more general result, Theorem \ref{main}, which connects a certain property of the Chow group of zero-cycles of a variety $V$ to a lower bound on the covering gonality of a very general plane section of $V$.

This paper is organized as follows. In Section \ref{basicProperties} we collect some useful properties of the covering gonality, including an alternate definition which is manifestly a birational invariant. Then in Section \ref{bounds} we prove Proposition \ref{BDELU} and Theorem \ref{main}, deriving Theorem \ref{mainCompleteIntersections}. Section \ref{comparisonProofs} contains proofs of some of the more involved formal properties of covering gonalities asserted in Section \ref{basicProperties}.
\subsection{Acknowledgements}
We would like to thank David Yang for helpful conversations throughout the project, especially the suggestion that \cite[Proposition 3.5]{RY16} could apply in the proof of Theorem \ref{main}. We would also like to thank Joe Harris for useful conversations.
\section{Formal properties of covering gonalities}\label{basicProperties}
In this section we prove some basic properties of covering gonalities. To help with this, we introduce an ancillary definition of covering gonality for arbitrary algebraic function fields.
\begin{definition}\label{cvgdef2}
Let $K$ be an algebraic function field over $k$. The \emph{covering gonality} of $K$, $\cvg_k(F)$, is the minimal $e$ such that there exists two fields $K_C$ and $K_B$ over $k$ and a diagram of finite field extensions
\begin{equation}\label{cvg2}
\begin{tikzcd}
K_C \arrow[r, hookleftarrow, "f"] 
&K\\
K_B(t)  \arrow[u,hookrightarrow,"\pi"]
\end{tikzcd},
\end{equation}
such that $[K_C:K_B(t)]=e$. The \emph{separable covering gonality} of $K$ is the minimal such $e$ such that the diagram (\ref{cvg2}) exists with $f$ separable.
\end{definition}
As the name suggests, this invariant is closely connected with the covering gonality.
\begin{proposition}\label{comparison}
If $X$ is a proper irreducible variety defined over the field $k$ with function field $K$, then $\cvg(X)=\cvg_k(K)$ and  $\scvg_k(X)=\scvg_k(K)$.
\end{proposition}

Working with covering gonalities for algebraic function fields simplifies some of our discussion. In particular, it makes base change more clear.
\begin{proposition}\label{fieldInvariance}
Let $K$ be an algebraic function field over a field $k$. If $L/k$ is an arbitrary field extension of $k$ and $K'$ a field in $K \otimes_k L$, then $\cvg_k(K)=\cvg_L(K')$. If $L/k$ is contained in a separable extension of the field of constants of $K/k$, then $\scvg_k(K)=\scvg_L(K')$.
\end{proposition}
The assumption that $L/k$ be contained in a separable extension of the field of constants of $K$ is necessary; in the unirational hypersurfaces of Example \ref{shioda}, the separable covering gonality of $X$ can be arbitrarily large by Proposition \ref{BDELU}, whereas the separable covering gonality of the rational space covering it is 1. 

Together, Propositions \ref{comparison} and \ref{fieldInvariance} indicate the first key properties of the covering gonality; it is a \emph{birational invariant} of $X$ and (for the most part) does not depend on the field of definition $k$. Their proofs are straightforward, and we defer them to Section \ref{comparisonProofs}.

We collect here some useful properties of covering gonalities, which echo comparable properties for the invariants defined in \cite{BDELU17} over $\CC$.%I'll figure this out

\begin{proposition}\label{properties}
\begin{enumerate}
\item Both $\cvg(X)$ and $\scvg_k(X)$ are birational invariants of $X$. 
 \item Given an extension $L/k$, if $X_L$ is any irreducible component of $X\times_k \mathrm{Spec}(L)$, then $\cvg(X)=\cvg(X_L)$. If $L/k$ is contained in a separable extension of the field of constants of $X$, then $\scvg(X)=\scvg(X_L)$.
\item $X$ satisfies $\cvg(X)=1$ if and only if $X$ is uniruled. X satisfies $\scvg_k(X)=1$ if and only if $X$ is separably uniruled over $k$.
\item If $X$ is a smooth irreducible projective curve and $k$ is algebraically closed, then $\scvg_k(X)=\mathrm{cvg}(X)=\mathrm{gon}_k(X)$.
\item If $k$ is uncountable and algebraically closed and $X$ is projective, then $\cvg(X)$ is equal to the minimum gonality of a curve through a general point of $X$. 
\item If $f:X\dashrightarrow Y$ is a dominant generically finite rational map of projective varieties over $k$, 
then $\cvg(Y)\leq \cvg(X)$.
If $f$ is also separable, then $\mathrm{scvg}_k(Y)\leq \mathrm{scvg}_k(X)$.

\end{enumerate}
\end{proposition}
\begin{proof}
Proposition \ref{comparison} immediately implies (1), because birational varieties have the same function fields. (2) follows from Proposition \ref{fieldInvariance}.

We observe that if $X$ has covering gonality $1$, then the diagram (\ref{cvg}) in fact gives a dominant generically finite morphism from $\PP^1\times B$ to $X$, so $X$ is uniruled. The same observation holds for separable covering gonality, proving (3).

(4) follows from Proposition A.1.vii of \cite{Poo07}.

The proof of (5) is analogous to the proof of Proposition IV.1.3.5 of \cite{Kol96}. Fix some ample class $H$ on $X$. Let $e$ be the minimum gonality of a curve through a general point of $X$. Because $k$ is uncountable, there must be some fixed degree $d$ such that the set of points $x\in X$ contained in degree $d$ gonality $e$ curves is dense. There is a quasiprojective variety $B=\mathrm{Hom}_d(\PP^1, \mathrm{Sym}^e(X))$ parametrizing maps $\phi:\PP^1\ra \mathrm{Sym}^e(X))$ such that $\phi^*(H)$ has degree $d$. Moreover, we can take $\OC$ to be such that the diagram
\[
\begin{tikzcd}
\OC\arrow[r]\arrow[d]& X\times \mathrm{\Sym^{e-1}(X)}\arrow[d,"\pi"]\\
\PP^1\times B\arrow[r,"\mathrm{ev}"] &\mathrm{Sym}^e(X)
\end{tikzcd}
\]
is a fiber product, where $\pi$ is the obvious symmetrizing map $p,(p_1\cdots p_{e-1})\mapsto (p_1\cdots p_{e-1}p)$. Composing $\OC\ra X\times \mathrm{\Sym^{e-1}(X)}$ with the projection $X\times \mathrm{Sym}^{e-1}(X)\ra X$ gives a map $f:\OC\ra X$. By the assumption on $d$, some component of $\OC$ dominates $X$, and restricting to that component and its image in $\PP^1\times B$ gives the desired covering family of $X$.

For (6), we note that a cover $g:\mathcal{C}\ra X$ of $X$ by gonality $e$ curves can be generically composed with $f:X\dashrightarrow Y$ to give a covering family of $Y$, giving the desired inequalities.
\end{proof}

Finally, we note that one can make a couple additional assumptions about covering families. In particular, over algebraically closed fields, we may assume that the covering gonality is computed by a family satisfying the additional properties of covering families assumed in Definition 1.4 of \cite{BDELU17}.
\begin{proposition}\label{coveringFamiliesCanBeGood}
If $X$ is a smooth geometrically integral proper variety with (separable) covering gonality $e$ defined over the (separably) closed field $k$, then there exists a (separable) covering family $\pi:\OC\ra \PP^1\times B$ of $X$ of degree $e$ such that $B$ and $\pi:\OC\ra B$ are smooth and the general fiber $C_b$ of $\pi$ is birational to its image in $X$.
\end{proposition}
\begin{proof}
We prove the result for separable covering gonalities; the result for $\cvg$ is analogous. Let $\OC\ra B$ be a covering family of $X$ as in (\ref{cvg}). Since $\OC$ is irreducible and $f:\OC\ra X$ is separable, we have that $\OC$ is geometrically integral, so $B$ is as well. Let $\OC'$ be the closure of the image of $\OC$ in $B\times X$. We have that $\OC'$ is geometrically integral because $\OC$ is, and $\OC'\ra B$ is proper. Moreover, the generic fiber of $\OC'\ra B$ has gonality at most $e$ by Proposition \ref{properties} (5). So $\OC'$ possesses a  rational map to $\PP^1$ such that $\OC'\dashrightarrow B\times \PP^1$ has degree at most $e$. Removing the image of the indeterminacy locus from $B$, we then have a family $\OC'\ra B\times \PP^1$ such that \emph{every} fiber over $B$ is isomorphic to its image in $X$. Let $\OC''$ be the normalization of $\OC'$. This variety is smooth in codimension 1, and so the map $\OC''\ra B$ is generically smooth, so restricting $B$ to some open set gives a smooth family of curves over smooth variety,
%this is true by Chevalley, but I'm ok with leaving it in this form.
 such that the general fiber $C_b$ is birational to its image in $X$.
\end{proof}
\section{Bounding covering gonalities}\label{bounds}
The two notions of the covering gonality given in this article satisfy similar formal properties to each other, as we saw above. But as a practical matter, existing methods of computing covering gonality mostly compute $\scvg$ in positive characteristic. For instance, Proposition \ref{BDELU} was first stated and proved in \cite{BDELU17} for the covering gonality of complex varieties, but its proof holds in positive characteristic, essentially as is. 
%
%The two notions of covering gonality defined above agree with each other in characteristic 0 and agree with the definition of \cite{BDELU17} when $k\cong \CC$. But the two notions can differ in any positive characteristic. In particular, smooth surfaces of large degree in $\PP^3$ can even be unirational \cite{Shi74}, and thus have fixed covering gonality $1$, but for 

\begin{proof}[Proof of Proposition \ref{BDELU}]
By Proposition \ref{properties} (2), we may assume the field $k$ is separably closed. By adjunction we have that the canonical bundle $\omega_X$ on $X$ is given by $\omega_X\cong \OO(mH)$, where $m:=\sum_i d_i -N-1$. Suppose $X$ has separable covering gonality $e$, and that this gonality is achieved by the separable covering family
\[
\begin{tikzcd}
\OC \arrow[r, "f"] \arrow[d, "\pi"]
& X \\
\PP^1\times B 
\end{tikzcd}
\]
satisfying all the properties listed in Proposition \ref{coveringFamiliesCanBeGood}. By adjunction, we have the isomorphism
\[
\omega_{\OC}\cong f^*(\omega_X)\otimes \OO(R),
\]
where $R$ is the ramification divisor of $f$.
 Let $b\in B$ be general, so the fiber $C_b$ of $\OC$ over $b$ is smooth and birational to its image in $X$. By adjunction, we have $\omega_{C_b}\cong \omega_{\OC}\vert_{C_b}$. Moreover, since $b$ is general, we have that $C_b$ is not contained in the ramification locus of $f$, so we have an isomorphism
 \[
 \omega_{C_b} \cong f^*(\omega_X)\vert_{C_b}+R',
 \]
 where $R'$ is some effective divisor on $C_b$. Let $D=p_1+\ldots+p_e$ be a general fiber of the map $\pi:C_b\ra \PP^1$, so in particular we may assume every $p_i$ is contained in the open subset of $C_b$ isomorphic to its image in $X$. By the Riemann-Roch theorem, we have $H^0(C_b,\omega_{C_b}(-p_1-\ldots-p_e)>H^0(C_b,\omega_{C_b})-e$, so for some $j\leq e-1$ we must have an equality 
 \[
 H^0(C_b, \omega_{C_b}(-p_1-\cdots-p_j))=H^0(C_b, \omega_{C_b}(-p_1-\cdots-p_{j+1})).
 \]
 But, supposing $j\leq m$, there is a degree $m$ hypersurface in $\PP^N$ passing through $f(p_1),\ldots,f(p_j)$ but not $f(p_{j+1})$, so the points $p_1,\ldots p_{j+1}$ impose independent conditions on $f^*(\omega_X)\vert_{C_b}$ and hence on $\omega_{C_b}$. So the equality above implies $e-1>m$, giving
 \[
 \scvg_k(X)\geq \sum_i d_i-N+1,
 \]
 as was to be shown.
\end{proof}
As Example \ref{shioda} demonstrates, we cannot hope that Proposition \ref{BDELU} holds for the ordinary covering gonality.
 However, over the course of the remainder of this section, we will use a different, Chow-theoretic method to prove the lower bound, Theorem \ref{mainCompleteIntersections}.
 In fact, we will prove a somewhat more general statement, Theorem \ref{main}, from which we will derive Theorem \ref{mainCompleteIntersections} using some known results.
 To state the theorem, we first introduce a  convenient notion.
\begin{definition}\label{hyp}
An irreducible proper variety $X$ defined over an algebraically closed field is \emph{Chow nondegenerate} if for all divisors $Z$ on $X$ the map of Chow groups of zero-cycles
\[
\mathrm{CH}_0(Z)\ra \mathrm{CH}_0(X)
\]
induced by inclusion is not surjective. We will call $X$ \emph{Chow degenerate} if it is not Chow nondegenerate.
\end{definition}
Uniruled varieties are always Chow degenerate, as given a uniruling $f:\PP^1\times B\dashrightarrow X$ and any closed point $p\in \PP^1$, we have that $\text{CH}_0(X)$ is generated by points in the union of the image of $p\times B$ and the exceptional locus of $f$ in $X$. So Chow nondegenerate varieties always satisfy $\cvg(X)\geq 2$. Riedl and Woolf \cite{RW18} observed this, and to give examples of non-uniruled varieties they proved the following.
\begin{theorem}[\cite{RW18}]\label{RiedlWoolf}
If $X/\overline{k}$ is a general complete interesection of multidegree $(d_1,\ldots,d_k)$ in $\PP^N$ and $\sum_i d_i-N-1\geq 0$, then $X$ is Chow nondegenerate.
\end{theorem}
\begin{remark}
This theorem in this exact format does not appear in \cite{RW18}, but follows from its Theorem 3.3---a result on the coniveau filtration of complete intersections originally due to \cite[Expos\'e XXI]{SGA7.2}---and Proposition 3.7.
\end{remark}
What we will show is that a very general plane section of a Chow nondegenerate variety has a covering gonality that increases with the codimension. More precisely, we prove the following.
\begin{theorem}\label{main}
Let $X$ be a projective Chow nondegenerate variety over an uncountable field $k=\overline{k}$. Fix an embedding $X\hookrightarrow \PP^N$. Then a very general codimension $c$ plane section of $X$ has covering gonality at least $\frac{3+c}{2}$.
\end{theorem}

%Riedl-Woolf in \cite{RW18} observe that uniruled varities are never Chow nondegenerate, as given a family of rational curves covering $X$, $f:\mathcal{C}\ra X$, the sum of an ample divisor and a divisor containing the locus not in the image of $f$ will generate $\mathrm{CH}_0(X)$. Using in an essential way some results from \cite{SGA7.2}, Expos\'e XXI on the coniveau filtration on \'etale cohomology of hypersurfaces in positive characteristic they then prove the following.
%\begin{theorem}[\cite{RW18}]\label{RiedlWoolf}
%If $X$ is a general complete intersection of multidegree $(d_1,\ldots,d_k)$ in  $\PP^N$ and $\sum_i d_i-N-1\geq 0$, then $X$ is Chow nondegenerate.
%\end{theorem}
Theorems \ref{main} and \ref{RiedlWoolf} immediately imply Theorem \ref{mainCompleteIntersections}, so it only remains to prove Theorem \ref{main}.

\begin{proof}[Proof of Theorem \ref{main}]
Fix some $e$. If the very general codimension $c$ plane section of $X$ has covering gonality at most $e$, then there is a family $\pi:\mathcal{C}\ra B $ of gonality $e$ curves on $X$ 
such that a general codimension $c$ plane section $X'$ of $X$ is dominated by fibers of $\pi$ that map entirely into $X'$. In particular, the general codimension $c$ plane section is dominated by a family of gonality at most $e$ curves of some fixed Hilbert polynomial $P_0(t)$. We will determine the lower bound on $e$ using this family.

We require an auxiliary parameter space to  accomplish this goal.
 For each $0\leq c'\leq c$, define $\mathcal{X}_{c'}$ as the set of triples $(\Lambda, Y,p)$, with $\Lambda$ a codimension $c'$ plane in $\PP^N$, $Y$ a divisor on $\Lambda$ of degree $2e-1$, and $p$ a geometric point of $(\Lambda \cap X)\setminus Y$. Let $R^0_{c'}$ be the subset of this locus containing triples $(\Lambda,Y,p)$ where there exists a reduced curve $C$ on $X\cap \Lambda$ with Hilbert polynomial $P_0(t)$ passing through $p$ such that $p$ is Chow-equivalent on $C$ to a divisor supported on $C\cap Y$, and let $R_{c'}$ be the closure of this locus.

 We first show that $R_c$ has codimension at most $2e-2$ in $\mathcal{X}_c$.  Fixing $\Lambda$ and $p$ general, there exists a curve $C$ through $p$ on $\Lambda \cap X$ with Hilbert polynomial $P_0(t)$ having a pencil of degree at most $e$ by hypothesis, and by the generality assumption on $p$ we may assume this pencil is unramified at $p$. Then the space of divisors $Y$ passing through the other $e-1$ points of the fiber of the pencil containing $p$ and entirely containing the $e$ points of some other fiber of the pencil, but not $p$, is a nonempty codimension $\leq 2e-2$ set in the space of all degree $2e-1$ divisors on $\Lambda$.

Now we bound the codimension of all the $R_{c'}$ from below using the fact that $X$ is Chow nondegenerate. To start, note that $R_0$ is a strict subset of $\mathcal{X}_0$. For given any divisor $Y$ on $X$, there is some $p$ on $X$ not Chow equivalent to any divisor on $Y$. In particular $p$ cannot be carried to $Y$ by a pencil on any $C$ with Hilbert polynomial $P_0(t)$ lying in $X$, or the flat limit of such curves. So the codimension of $R_0$ in $\mathcal{X}_0$ is at least $1$.

We now apply the following lemma, which is a variation of
 Proposition 3.5 of \cite{RY16}.
\begin{lemma}\label{riedlYang}
Assume $k$ is algebraically closed. Let $S_{c'}$ and $S_{c'+1}$ be closed subsets of $\mathcal{X}_{c'}$ and $\mathcal{X}_{c'+1}$ satisfying the following property:\begin{quote}
If $(\Lambda,Y,p)$ is in  $S_{c'+1}$, then for any codimension $c'$ plane section $\Lambda'$ of $\PP^N$ containing $\Lambda$ and divisor $Y'$ on $\Lambda'$ such that  $Y'\vert_{\Lambda}=Y$, we have $(\Lambda',Y',p)\in S_{c'}$.
\end{quote}
Then if every irreducible component of $S_{c'}$ has codimension at least $\epsilon>0$, every irreducible component of $S_{c'+1}$ has codimension at least $\epsilon+1$.
\end{lemma}
We have that each $R_{c'}$ and $R_{c'+1}$ satisfy the hypothesis of this lemma, so by induction, we have that $R_{c'}$ has codimension at least $c'+1$ in $\mathcal{X_{c'}}$. Working with the codimension of $R_c$, we then have $c+1\leq \mathrm{codim}(R_c)\leq  2e-2$, so $e\geq \frac{3+c}{2}$, as was to be shown.
\end{proof}
\begin{proof}[Proof of Lemma \ref{riedlYang}] 
We may clearly assume that $S_{c'+1}$ is irreducible. It suffices to work in the case where $S_{c'}$ is also irreducible, as given the irreducibility of $S_{c'+1}$, the locus of triples in $\mathcal{X}_{c'}$ with hyperplane section in $S_{c'+1}$ is itself irreducible. Say $S_{c'}$ has codimension $\epsilon$. Let $\Phi\subset S_{c'}\times S_{c'+1}$ be the incidence correspondence between $(\Lambda',Y',p)\in S_{c'+1}$ and $(\Lambda, Y, p)\in S_{c'}$ with $(\Lambda',Y',p)$ as a hyperplane section. 
$\Phi$ is irreducible as a fiber bundle with irreducible fibers over $S_{c'+1}$. It has dimension $\dim(S_{c'+1})+\dim(F)$, where $F$ is any fiber of the projection $\Phi\ra S_{c'+1}$. We note that there is an $n-c'-1$ dimensional family of hyperplane sections of a given triple $(\Lambda,Y,p)\in \mathcal{X}_{c'}$, so $\Phi$ has dimension strictly less than $\dim(S_{c'})+n-c'-1$, as by Lemma 3.6 of \cite{RY16} $S_{c'+1}$ cannot contain every hyperplane section of $S_{c'}$. So we have 
\[
\dim(S_{c'+1})< \dim(S_{c'})+n-c'-1-\dim(F).
\]
Using $\dim(F)+\dim(\mathcal{X}_{c'+1})=\dim(\mathcal{X}_{c'})+n-c'-1$, we then have
\[
\dim(S_{c'+1})< \dim(S_{c'})+\dim(\mathcal{X}_{c'+1})-\dim(\mathcal{X}_{c'})
\]
proving the desired result.
\end{proof}

\section{Covering gonality for algebraic function fields}\label{comparisonProofs}
%Definition \ref{cvgdef} is only given for varieties defined over algebraically or separably closed fields. It is natural to want a more general definition, and in this section we consider a definition defined more generally.
In this section we prove Propositions \ref{comparison} and \ref{fieldInvariance} on covering gonalities for algebraic function fields. In this section, we will use the term \emph{diagram of field extensions} of $K$ to refer to a diagram
\[
\begin{tikzcd}
K_C \arrow[r, hookleftarrow, "f"] 
&K\\
K_B(t)  \arrow[u,hookrightarrow,"\pi"]
\end{tikzcd}
\]
satisfying the properties of Definition \ref{cvgdef2}.

We first prove Proposition \ref{fieldInvariance} in two parts: we first show that (separable) covering gonalities are preserved after replacing $k$ with the field of constants of $K/k$ as Lemma \ref{constants}; we then show in Lemma \ref{arbitraryExtensions} that replacing that field with any (separable) extension preserves (separable) covering gonalities.

\begin{lemma}\label{constants}
If $k^c$ is the field of constants of $K/k$, then
\[
\cvg_k(K)=\cvg_{k^c}(K) \text{ and }\scvg_k(K)=\scvg_{k^c}(K).
\]

\end{lemma}
\begin{proof}
Let $K_C/K$ and $K_C/K_B(t)$ be extensions of fields as in Definition \ref{cvgdef2}. The map $f:K\hookrightarrow K_C$ restricted to $k^c$ gives $K_C$ the structure of a field extension of $k^c$.  Then the subfield of $K_C$ generated by $K_B+L$ is a field $K_{B'}$ defined over $k^c$, the natural map $K_{B'}(t)\hookrightarrow K_C$ is an extension of fields defined over $k^c$ and we have $[K_C:K_{B'}(t)]\leq [K_C:K_B(t)]$. So $\cvg_{k^c}(K)\leq \cvg_{k}(K)$ and $\scvg_{k^c}(K)\leq \scvg_{k}(K)$.

On the other hand, given any covering diagram (\ref{cvg2}) of field extensions of $K/k^c$, regarding $K_C$ and $K_B$ as field extensions of $k$ gives a diagram of field extensions of $K/k$ while preserving $[K_C:K_B(t)]$ and the separability of $K_C/K$, if applicable. So we have  $\cvg_{k^c}(K)\geq \cvg_{k}(K)$ and $\scvg_{k^c}(K)\geq \scvg_{k}(K)$. 
\end{proof}

\begin{lemma}\label{arbitraryExtensions}
Let $K/k$ be an algebraic function field with field of constants $k$. Then if $L$ is any field extension of $k$ then \[
\cvg_L(K\otimes_k L)=\cvg_k(K).
\]
If $L/k$ is separable, then
\[
\scvg_L(K\otimes_k L)=\scvg_k(K).
\]
\end{lemma}
\begin{proof}
We only prove the first assertion; the second follows analogously. Any extension of $k$ is an algebraic extension of a purely transcendental extension of $k$, so we handle those two cases separately. First, suppose that $L/k$ is algebraic. Given a diagram of field extensions
\[
\begin{tikzcd}
K_C \arrow[r, hookleftarrow, "f"] 
&K\\
K_B(t)  \arrow[u,hookrightarrow,"\pi"]
\end{tikzcd},
\]
defined over $k$, let $K_C'$ be a quotient of $K_C\otimes_k L$ by some maximal ideal, which will be 0 if $K_C$ has field of constants $k$.
 Then the diagram
\[
\begin{tikzcd}
K_C' \arrow[r, hookleftarrow, "f"] 
&K\otimes_k L\\
(K_B\otimes_k L\cap K_C')(t)  \arrow[u,hookrightarrow,"\pi"]
\end{tikzcd}
\]
is a diagram of extensions as in Definition \ref{cvgdef2} and we have
\[
[K_C':(K_B\otimes_k L\cap K_C')(t) ]\leq [K_C:K_B(t)]. \]
So $\cvg_L(K\otimes_k L)\leq \cvg_k(K)$. For the other direction, suppose 
\[
\begin{tikzcd}
K_C \arrow[r, hookleftarrow, "f"] 
&K\otimes_k L\\
K_B(t)  \arrow[u,hookrightarrow,"\pi"]
\end{tikzcd}
\]
is a diagram of extensions over $L$ as in Definition \ref{cvgdef2}. Because $K_C$ and $K_B$ are finitely generated, there is some finitely generated extension $L'/k$ contained in $L$ such that $K_B$ admits a presentation $K_B\cong L(b_1,\ldots,b_k)/(r_1,\ldots, r_\ell)$ where the $r_i$ are polynomials in the $b_j$ with coefficients in $L'$ and $K_C$ admits a presentation $K_B(t)(c_1,\ldots,c_{k'})/(q_1,\ldots,q_{\ell'})$ where the $q_i$ are polynomials in the $c_j$ with coefficients in $L'(b_1,\ldots,b_k,t)$. Moreover, we can assume that the generators $c_i$ include the image in $K_C$ of a set of generators of $K/k$. 

 Let $K_B':=L'(b_1,\ldots,b_k)/(r_1,\ldots, r_\ell)$ and $K_C':=K'_B(t)(c_1,\ldots,c_{k'})/(q_1,\ldots,q_{\ell'})$. 
Since the $c_i$ include a set of generators of $K/k$, $K_C'$ is an extension of $K$, and $K_C'/K$ is finite because it is finitely generated and $K_C'$ is contained in the algebraic extension $K_C$ of $K$. And, noting $K_C\cong K_c'\otimes_{L'} L$ and $K_B\cong K_B'\otimes_{L'} L$, we have $[K_C:K_B(t)]=[K_C':K_B(t)']$. So we have $\cvg_L(K\otimes_k L)\geq \cvg_k(K)$, hence $\cvg_L(K\otimes_k L)=\cvg_k(K)$.

Now suppose $L$ is purely transcendental over $k$. $\cvg_L(K\otimes_k L)\leq \cvg_k(K)$ is clear by the same argument as above. Now suppose we have a diagram
\[
\begin{tikzcd}
K_C \arrow[r, hookleftarrow, "f"] 
&K\otimes_k L\\
K_B(t)  \arrow[u,hookrightarrow,"\pi"]
\end{tikzcd}.
\]
By picking some finite sets of generators and relations for $K_C/K_B(t)$, $K_C/K$ and $K_B/L$, and replacing $L$ by the field generated over $k$ by all coefficients in the relations defining these extensions, we may assume $L$ is finitely generated, so inductively we may assume $L=k(x)$. In addition, by the result for algebraic extensions, we may assume that $k$ is infinite. So for the rest of the proof, we assume that $k$ is infinite and $L=k(x)$.

First, given a diagram of algebraic function fields over $k$ as in (\ref{cvg2}), taking the tensor product of the diagram with $k(x)$ gives a covering diagram of $K\otimes_k k(x)$ of the same degree as that for $K$. So $\cvg_k(K)\geq \cvg_{k(x)}(K\otimes_k k(x))$.

In the other direction, first note that we can assume $k$ is infinite by replacing it as needed by some infinite extension; this does not change covering gonalities by the result for algebraic extensions above. Now suppose we have a covering diagram of $K\otimes_k k(x)$ over $k(x)$ as in (\ref{cvg2}). Let $\OC$, $B$, and $X$ be varieties over $k$ with function fields $K_C$, $K_B$ and $K$, so we have a diagram of rational maps
\[
\begin{tikzcd}
\OC \arrow[r,dashrightarrow, "f"] \arrow[d,dashrightarrow, "\pi"]
& X\times \PP^1 \\
\PP^1\times B 
\end{tikzcd}.
\]
Replacing $\OC$ by a open subvariety as needed, we may assume $f$ and $\pi$ are regular. Moreover, $\OC$ and $B$ both possess compatible rational maps to $\PP^1$, where on $\OC$ the map is given by composing $f$ with the projection $X\times \PP^1\ra \PP^1$. Because $f$ and $\pi$ are generically finite, for all but finitely many $k$-points $p$ in $\PP^1$ the maps of fibers 
\[
\begin{tikzcd}
\OC_p \arrow[r, "f"] \arrow[d, "\pi"]
& X \\
\PP^1\times B_p 
\end{tikzcd}
\]
are generically finite and $B_p$ and $\OC_p$ are both integral, and moreover the restriction of $\pi$ to $\OC_p\ra \PP^1\times B_p$ has unchanged degree. The function fields of one of these fibers form a covering diagram of $X/k$ of degree equal to $[K_C:K_B(t)]$, so $\cvg_k(K)\leq \cvg_{k(x)}(K\otimes_k k(x))$.
\end{proof}
\begin{proof}[Proof of Proposition \ref{fieldInvariance}]
Suppose $L/k$ is an extension of $k$ contained in a separable extension of $k^c$. Let $L^c$ be the field of constants of $K\otimes_{k^c\cap L} L$ over $L$, so $L^c$ is separable over $k^c$. Then by Lemmas \ref{constants} and \ref{arbitraryExtensions} we have
\[
\scvg_k(K)=\scvg_{k^c}(K)=\scvg_{L^c}(K\otimes_{k^c}L^c)=\scvg_{L}(K\otimes_{k^c}L^c)
\]
and we have $K\otimes_{k^c} L^c$ is isomorphic to any field contained in $K\otimes_{k} L$. For covering gonalities, suppose $L$ is an arbitrary extension of $k$. Letting $k^c$ and $L^c$ be as above, by Lemmas \ref{constants} and \ref{arbitraryExtensions} we have the chain of equalities
\[
\cvg_k(K)=\cvg_{k^c}(K)=\cvg_L(K \otimes_{k_c} L^c),
\]
proving the result.
\end{proof}
We now prove Proposition \ref{comparison}. We restate it here for reference.
\begin{manualProposition}{\ref{comparison}}
If $X$ is a proper irreducible variety defined over the field $k$ with function field $K$, then $\cvg(X)=\cvg_k(K)$ and  $\scvg_k(X)=\scvg_k(K)$.
\end{manualProposition}
\begin{proof}
One direction here is clear; if we have a diagram of irreducible varieties
\[
\begin{tikzcd}
\OC \arrow[r, "f"] \arrow[d, "\pi"]
& X \\
\PP^1\times B 
\end{tikzcd}
\]
as in Definition \ref{cvgdef}, taking function fields of every variety gives a covering diagram of $K$ in the sense of Definition \ref{cvgdef2}. So $\cvg(X)\geq \cvg_k(K)$ and $\scvg(X)\geq \scvg_k(K)$.

We now prove the inequality $\cvg(X)\leq \cvg_k(K)$. Suppose we have a diagram of field extensions of $K$,
\begin{equation*}\label{placeholder1}
\begin{tikzcd}
K_C \arrow[r, hookleftarrow, "f"] 
&K\\
K_B(t)  \arrow[u,hookrightarrow,"\pi"]
\end{tikzcd}.
\end{equation*}
 We will construct $\OC$ and $B$ with function fields $K_C$ and $K_B$ respectively that fit into a covering diagram (\ref{cvg}) of $X$ that produces the above diagram upon taking function fields.

Let $\OC_\eta$ be the unique regular projective curve defined over $K_B$ with function field $K_C$. Then over some variety $B'/k$ with function field $K_B$ there is a variety $\OC'$, a proper map $\pi:\OC'\ra B'$ and a fiber diagram
\[
\begin{tikzcd}
\OC_\eta\arrow[r, hookrightarrow] \arrow[d, "\pi"] &\OC'\arrow[d,"\pi"]\\
\mathrm{Spec}(K_B)\arrow[r,hookrightarrow]&B'.
\end{tikzcd}
\]
This can be shown by embedding $C_\eta$ in some projective space over $K_B$ and letting $B'$ be affine with fraction field $K_B$ such that $C_\eta$ is cut out by homogeneous polynomials with coefficients in $H^0(B',\OO_{B'})$. Moreover, since $\OC'$ has function field $K_B$, the function $t$ on $\OC'$ defines a rational map $\OC'\dashrightarrow \PP^1$. This is defined except on some closed subset $Z$ of $\OC'$ of codimension at least 2.
 In addition, let $Z'$ be the locus in $B'$ of $b$ such that $t$ is constant on some irreducible component of $\OC'\vert_{\pi^{-1}(b)}$; this is not all of $B$ since $t$ is transcendental over $K_B$. Define $B''=B'\setminus (\pi(Z)\cup Z')$, and let $\OC''$ be the open subset of $\OC'$ lying over $B''$. 
Finally, we have a rational map $f:\OC''\dashrightarrow X$, which is again defined except on a closed subset $Z''$ of codimension at least 2 on $\OC''$.
Then, defining $B=B''\setminus \pi(Z'')$ and $\OC=\pi^{-1}(B)$, we have exactly the desired diagram
\[
\begin{tikzcd}
\OC \arrow[r, "f"] \arrow[d, "\pi"]
& X \\
\PP^1\times B 
\end{tikzcd},
\]
where both $\pi$ and $f$ are regular by the construction, and $\pi$ is finite surjective because we removed the locus $Z'$ over which it failed to be finite surjective. Since $\pi$ is induced by the field extension $K_C/K_B(t)$, it has the same degree, and we conclude $\cvg(X)\leq \cvg_k(K)$, giving the desired equality.

If we, throughout this process, assumed $f$ was separable, no aspect of this construction would change, so we also have $\scvg_k(X)\leq \scvg_k(K)$.
\end{proof}
\begin{comment}
\section{Stuff}
Uniruledness is a closed condition by Corollary IV.1.5.1 of \cite{Kol96}. In this optimistic %lol
section we hope to generalize to proving lower semicontinuity of $\cvg$ in families

\begin{proposition}
Let $R$ be a DVR with fraction field $K$ and residue field $k$, and let $f:X \ra \mathrm{Spec}(R)$ be a map with $X$ normal and irreducible and $X_K$ projective. %insert futher hypotheses
Then given any component $Y$ of $X_k$ we have $\mathrm{cvg}_k(Y) \leq \mathrm{cvg}_K(X_K)$.
\end{proposition}
\begin{proposition}
$\cvg(X)$ is equivalently the minimal $e$ such that $k(X)$ admits a finite extension $K$ which is itself a finite degree $e$ extension of a field of the form $L(t)$, where $L/k$ has transcendence degree $n-1$. $\scvg_k(X)$ is given my the minimum $e$ such that this exists with $K/k(X)$ separabale.
\end{proposition}
\begin{proof}
One direction here is obvious, as the diagram (\ref{cvg}) gives the desired extensions upon taking the fields of rational functions. For the other direction, let $B'$ be a variety with field of fractions $L$ and let $\mathcal{C}_\eta$ be a smooth projective curve over $L$ with field of fractions $K$.  Let $\mathcal{C}\ra B'$ be a map of curves with generic fiber...

Let $\mathcal{C}'\ra B'\times \PP^1$ be a morphism representing the extension $K/L(t)$. The generic fiber of the map $\mathcal{C}'\ra B$ can be modified to a smooth projective curve, so the general fiber can be as well. So, replacing $B'$ by a smaller variety $B$ as needed, we can produce the desired diagram
\end{proof}
\end{comment}

\bibliographystyle{plain}
\bibliography{Irrationality}
\end{document}